\documentclass[12pt]{article}

\usepackage{graphicx,amsmath,amscd,amsfonts,amsthm,amssymb,verbatim,stmaryrd,fullpage}
\newtheorem{theorem}{Theorem}
\newtheorem{lemma}[theorem]{Lemma}
\newtheorem{prop}[theorem]{Proposition}

\newtheorem{definition}[theorem]{Definition}

\newtheorem{conj}{Conjecture}

\theoremstyle{remark}

\newcommand{\R}{\mathbb R}
\newcommand{\C}{\mathbb C}

\newcommand{\Z}{\mathbb Z}

\newcommand{\p}{\mathfrak p}

\newcommand{\g}{\mathfrak g}

\newcommand{\Aa}{\mathfrak a}

\newcommand{\ssl}{\mathfrak{sl} }

\newcommand{\Op}{\text{Op} }
\newcommand{\OO}{\mathcal O}
\newcommand{\W}{\mathcal W}

\newcommand{\HH}{\mathbb H}
\newcommand{\Ai}{\mathcal A}
\newcommand{\Bi}{\mathcal B}

\newcommand{\Ss}{\mathcal S}
\newcommand{\Tt}{\mathcal T}

\newcommand{\Ii}{\mathcal I}

\newcommand{\Ad}{\text{Ad}}
\newcommand{\End}{\text{End}}

\newcommand{\tr}{\text{tr}}

\newcommand{\Sym}{\text{Sym}}
\newcommand{\sym}{\text{sym}}
\newcommand{\Hom}{\text{Hom}}

\newcommand{\ts}{\tilde{s} }

\begin{document}

\title{Triple Product $L$ Functions and Quantum Chaos on $SL(2,\C)$}

\author{Simon Marshall}

\maketitle

\begin{abstract}
We extend the results of Watson, which link quantum unique ergodicity on arithmetic hyperbolic surfaces with subconvexity for the triple product $L$ function, to the case of arithmetic hyperbolic three manifolds.  We work with the full unitary dual of $SL(2,\C)$, and consider QUE for automorphic forms of arbitrary fixed weight and growing spectral parameter.  We obtain our results by constructing microlocal lifts of nonspherical automorphic forms using representation theory, and quantifying the generalised triple product formula of Ichino in the case of complex places.
\end{abstract}

\section{Introduction}

If $M$ is a compact Riemannian manifold, it is a central problem in quantum chaos to understand the behvaiour of high energy Laplace eigenfunctions on $M$.  If $\{ \psi_n \}$ is a sequence of such eigenfunctions with eigenvalues $\lambda_n$ tending to $\infty$, a natural question that one may ask is whether $\psi_n$ are becoming approximately constant.  This may be asked either in a pointwise sense, by showing that certain $L^p$ norms of $\psi_n$ are small, or on average, by showing that the probability measures $\mu_n = |\psi_n(x)|^2 dv$ tend weakly to the Riemannian volume $dv$ of $M$.  There is a conjecture of Rudnick and Sarnak \cite{RS} known as the quantum unique ergodicity conjecture, or QUE, which predicts the equidistribution of $\psi_n$ in this weak-* sense when $M$ is negatively curved, and in this paper we shall be interesed in a case of this conjecture in which $M$ is an arithmetic hyperbolic three manifold.

The QUE conjecture predicts not just the equidistribution of $\mu_n$, which can be thought of as the positions of the quantum states $\psi_n$, but of a semiclassical analogue of the combined position and momentum called the microlocal lift.  This is most naturally described in terms of the correspondence between the geodesic flow on $S^*M$ and the space $L^2(M)$ with the unitary Schr\"odinger evolution.  The observables of the geodesic flow are functions $a \in C^\infty( S^*M)$, and after making a choice of quantisation scheme it is possible to associate to each classical observable $a$ a self-adjoint operator $\text{Op}(a)$, which may be thought of as a quantum observable taking the value $\langle \text{Op}(a) \psi, \psi \rangle$ on a wavefunction $\psi \in L^2(M)$.  For each fixed $\psi$ the map

\begin{equation}
\label{stdlift}
\widetilde{\mu}_\psi(a): a \mapsto \langle \text{Op}(a) \psi, \psi \rangle
\end{equation}

can be shown to be a distribution on $C^\infty( S^*M)$, and this is defined to be the microlocal lift of $\psi$.  This construction is due to \v{S}nirel'man \cite{Sn}, who also proved that any high energy weak limit of the $\widetilde{\mu}_\psi$ is a measure invariant under the geodesic flow.  The microlocal form of the QUE conjecture then predicts that the only limit of $\{ \widetilde{\mu}_n \}$ is Liouville measure, or that all derivatives of $\psi_n$ are behaving randomly.

The fact that weak limits of the $\widetilde{\mu}_n$ are flow invariant measures makes it possible to apply ergodic techniques to the QUE conjecture, which has lead to an essentially complete solution in the case of arithmetic quotients of $\HH^2$ and $( \HH^2 )^n$ by Lindenstrauss \cite{Li1, Li2} (with contributions by Soundararajan \cite{So1} to deal with the noncompact case), and compact quotients of $GL(n,\R)$ for $n$ prime by Silberman and Venkatesh \cite{SV, SV2}.  In the case of arithmetic quotients of $\HH^2$, a second approach based on the triple product $L$ function was developed by Watson \cite{W}.  To give an illustration of his results, let $X = SL(2,\Z) \backslash \HH^2$ and let $\phi_i$ be three $L^2$ normalised Hecke-Maass cusp forms on $X$ with associated representations $\pi_i$.  Watson then proves the beautiful identity

\begin{equation}
\label{watson}
\left| \int_X \phi_1 \phi_2 \phi_3 dv \right|^2 = \frac{1}{8} \frac{ \Lambda( 1/2, \pi_1 \otimes \pi_2 \otimes \pi_3 ) }{ \prod \Lambda ( 1, \Ad \pi_i ) }.
\end{equation}

A consequence of this formula is that the coarse form of the QUE conjecture would be implied by a subconvex bound for the triple product $L$ function in the eigenvalue aspect, and similar formulae for vectors of higher weight in $\pi_i$ would allow one to deduce the full microlocal version.  The purpose of this paper is to prove the same implication in $\HH^3$ for a standard collection of manifolds $Y$ called the Bianchi manifolds.

As we are considering the QUE conjecture from the point of view of automorphic forms, it is natural for us to work not only with Laplace eigenfunctions and their associated spherical representations but with the full unitary dual of $SL(2,\C)$.  This dual is indexed by a weight $k \in \Z$ and and a spectral parameter $r \in \R$ (ignoring the complimentary series), and for each fixed $k$ we shall study the quantum limits of automorphic forms of weight $k$ and growing spctral parameter.  To describe our approach to this, let $\{ \pi_n \}$ be a sequence of automorphic representations of this type.  We may associate to $\{ \pi_n \}$ a sequence of sections $\{ \psi_n \}$ of a principal $SU(2)$ bundle over $Y = \Gamma \backslash SL(2,\C) / SU(2)$, which are the objects whose high energy behaviour we shall study.  We first construct microlocal lifts of $\psi_n$ in terms of the representations $\pi_n$, following eariler constructions of Zelditch \cite{Ze1, Ze2}, Lindenstrauss \cite{Li2}, and Silberman and Venkatesh \cite{SV}, and provide heuristics to illustrate what the generalisation of the QUE conjecture to these vector valued objects should be.  Our construction has the novel feature that we are applying it to nonspherical representations, and as a result of this and the nonabelianness of $SU(2)$ we find that there is a richer set of expected quantum limits than in the case of functions.  For instance, when the QUE conjecture for $\psi_n$ is interpreted in terms of differential forms on $Y$ it predicts different quantum limits for exact and coclosed 1-forms.  Once we have defined a lift in terms of automorphic forms we may then test its convergence to the expected limit by integration, using the general $GL_2$ triple product formula of Ichino \cite{I}.  The expected equivalence between QUE and subconvexity follows from Ichino's formula once once we have made it sufficiently quantitative, which requires the estimation of certain Archimedean local integrals using the Whittaker function formulas of Jacquet-Langlands \cite{JL} and a formula appearing in a paper of Michel and Venkatesh \cite{MV}.

The structure of the paper is as follows.  We introduce our notation in section \ref{notation}, before giving precise statements of our results in section \ref{definition}.  We establish the basic properties of our microlocal lift in sections \ref{limits} and \ref{agreement0}, and establish the relationship between QUE and subconvexity in sections \ref{tripleprod} and \ref{archimedean}.  We conclude in section \ref{modQUE} by giving an interpretation of the vector-valued QUE conjecture in terms of differential forms on $Y$.

{\bf Acknowledgements}: We would like to thank our adviser Peter Sarnak for suggesting this problem as part of our thesis, and providing much guidance and encouragement in the course of our work.

\section{Notation}
\label{notation}

Let $F$ be an imaginary quadratic field, which we assume for simplicity to have class number one, with ring of integers $\OO$.  Let $G = SL(2,\C)$ and $K = SU(2)$, let $\Gamma \subset G$ be the projection of $GL(2,\OO)$ to $G$ by central twisting, and let $Y = \Gamma \backslash G / K$ be a Bianchi manifold.  Note that $Y$ is also equal to $GL(2,\OO) Z \backslash GL(2,\C) / K$.  We begin by establishing notation for representations of $K$ and $G$.  Let $\rho_m$ denote the irreducible $m+1$ dimensional representation of $K$ with Hermitian inner product $\langle \: , \: \rangle$, and let $\cdot^*$ denote the associated conjugate linear isomorphism between $\rho_m$ and $\rho^*_m$.  We choose an orthonormal basis $\{ v_t \}$ ($t = m, m-2, \ldots, -m$) for $\rho_m$ and dual basis $\{ v_t^* \}$ for $\rho^*_m$, consisting of eigenvectors of $M$ satisfying

\begin{equation*}
\left( \begin{array}{cc} e^{i\theta} & 0 \\ 0 & e^{-i\theta} \end{array} \right) v_t = e^{i t\theta} v_t, \quad \left( \begin{array}{cc} e^{i\theta} & 0 \\ 0 & e^{-i\theta} \end{array} \right) v_t^* = e^{-i t\theta} v_t^*.
\end{equation*}

Let

\begin{eqnarray*}
H = \left( \begin{array}{cc} 1 & 0 \\ 0 & -1 \end{array} \right), \quad X_+ = \left( \begin{array}{cc} 0 & 1 \\ 0 & 0 \end{array} \right), \quad X_- = \left( \begin{array}{cc} 0 & 0 \\ 1 & 0 \end{array} \right), 
\end{eqnarray*}

$T = iH$, $Y_+ = iX_+$, $Y_- = -i X_-$ be a basis for the real lie algebra $\ssl(2,\C)$, and let $H^*, X_+^*$ etc. denote the elements of the dual basis.  Let $\mathfrak{h} = \langle H, T \rangle$ be the Cartan subalgebra of $\ssl(2,\C)$.  The Casimir operator is given by 

\begin{equation*}
4C = H^2 - T^2 + 2X_+X_- + 2X_-X_+ + 2Y_+Y_- + 2Y_-Y_+.
\end{equation*}

If $r \in \C$ and $k \in \Z$, let $\lambda = 2ir H^* + ik T^* \in \mathfrak{h}^*$ and let $I_\lambda$ be the representation of $SL(2,\C)$ unitarily induced from the character

\begin{equation*}
\chi: \left( \begin{array}{cc} z & x \\ 0 & z^{-1} \end{array} \right) \mapsto (z / |z| )^k |z|^{2ir}.
\end{equation*}

These are unitarisable for $(k,r)$ in the set

\begin{equation*}
U = \{ (k,r) | r \in \R \} \cup \{ (k,r) | k = 0, r \in i(-1,1) \},
\end{equation*}

and two such representations $I_\lambda$, $I_\mu$ are equivalent iff $\lambda = \pm \mu$.  Furthermore, these are all the irreducible unitary representations of $SL(2,\C)$ other then the trivial representation.  We choose a set $U' \subset U$ representing every equivalence class in $U$ to be

\begin{equation*}
U' = \{ (k,r) | r \in (0,\infty) \} \cup \{ (k,r) | r = 0, k \ge 0 \} \cup \{ (k,r) | k = 0, r \in i(0,1) \}.
\end{equation*}

Given $\pi \in \widehat{G}$ nontrivial, we shall say $\pi$ has weight $k$ and spectral parameter $r$ if it is isomorphic to $I_\lambda$ with $\lambda = 2ir H^* + ik T^*$, $(k,r) \in U'$.  For fixed $k$, all representations of weight $k$ may be realised on the space

\begin{equation*}
W_k = \{ f \in L^2(K) | f(mg) = \chi_k(m) f(g), m \in M \},
\end{equation*}

where in practice we shall assume the weight to be understood and denote this space by $W$, with $W_K$ denoting the subspace of $K$-finite vectors.  Let $T$ be the conjugate linear mapping $W \rightarrow W^*$, $T: f \mapsto \langle \cdot ,f \rangle$.  If $|k| \le m$, $I_\lambda$ will contain $\rho_m$ (or $\rho_m^*$, as they are isomorphic) as a $K$-type with multiplicity one, and we choose explicit unitary embeddings $\rho_m \rightarrow W$ and $\rho_m^* \rightarrow W$ using our choice of basis by

\begin{eqnarray}
v & \mapsto & (m+1)^{1/2} \langle \rho_m(k) v, v_m \rangle, \\
v^* & \mapsto & (m+1)^{1/2} \langle \rho_m^*(k) v^*, v_{-m}^* \rangle, \\
\label{psi}
\psi_j(k)^* & = & \langle \rho_m^*(k) v_j^*, v_{-m}^* \rangle.
\end{eqnarray}

If $v \in W$ we shall often think of $v$ as a vector in all representations $I_\lambda$ of weight $k$ simultaneously, as for $v$ and $v^*$ under this embedding.

Fix an $m$ and let $\rho = \rho_m$; we shall define the correspondence between sections of $X \times_K \rho$ and automorphic forms.  Recall that for a representation $\tau$ of $K$, the principal bundle $X \times_K \tau$ is the quotient of $X \times \tau$ by the right $K$-action

\begin{equation}
\label{prin}
(x,v)k = (xk, \rho(k)^{-1} v),
\end{equation}

so that sections of $X \times_K \tau$ may be thought of as sections of $X \times \tau$ satisfying

\begin{equation*}
\rho(k) v(xk) = v(x).
\end{equation*}

Let $\sigma = \rho \otimes \rho^*$ and define the bundles $B = X \times_K \rho$ and $E = \End(B) = X \times_K \sigma$, with the Hermitian structures coming from the one on $\rho$.  There is an equivalence between square integrable sections $s \in L^2(Y,B)$  of $B$ and $K$-homomorphisms $\rho^* \rightarrow L^2(X)$ via the map

\begin{equation}
\label{section}
s \mapsto ( v \mapsto ( s(x), v) ),
\end{equation}

and so the decomposition of $L^2(X)$ as a direct integral of automorphic representations induces one of $\Hom_K( \rho^*, L^2(X))$ and $L^2(Y,B)$.  Elements of $L^2(Y,B)$ corresponding to the discrete spectrum will be called automorphic sections, and these are the analogues of Laplace eigenfunctions for which our lift will be defined (in particular we ignore the continuous spectrum of $X$).  If $\pi \subset L^2(X)$ occurs discretely and $R_\pi$ is the unitary embedding $W \rightarrow L^2(X)$ associated to $\pi$, the definition of the section $s$ associated to $\pi$ by (\ref{section}) may be unwound to give

\begin{equation*}
s = (m+1)^{1/2}\sum_{i=0}^m R_\pi( \psi_{m-2i}^* ) v_{m-2i}.
\end{equation*}

We will $L^2$ normalise this, so that our definition of the section $s$ associated to a representation $\pi$ is

\begin{equation}
\label{section2}
s = \sum_{i=0}^m R_\pi( \psi_{m-2i}^* ) v_{m-2i}.
\end{equation}

There is a natural identification of $X$ with the orthonormal frame bundle of $Y$ and of $S^*Y$ with $X / M$, and so if $\pi : S^*Y \rightarrow Y$ is the projection this induces isomorphisms of $\pi^*(B)$ and $\pi^*(E)$ with $X \times_M \rho$ and $X \times_M \sigma$, which we shall implicity make use of throughout the paper.

\section{Statement of Results}
\label{definition}

Fix a representation $\rho = \rho_m$ and a weight $k$ with $|k| \le m$, and consider an automorphic section $s \in L^2(Y,B)$ of $B = X \times_K \rho$ associated to a representation $\pi$ of weight $k$.  Our first result is the construction of a microlocal lift of $s$ in terms of $\pi$.  Explicit lifts of this kind have already been constructed for functions on an arbitrary locally symmetric space $Y = \Gamma \backslash G / K$, which was first carried out by Zelditch \cite{Ze1, Ze2} for $G = SL(2,\R)$, before being extended to $SL(2,\R)^n$ by Lindenstrauss \cite{Li2} and to arbitrary semisimple Lie groups by Silberman and Venkatesh \cite{SV}.  These require the function to be an eigenfunction of the full ring of invariant differential operators on $Y$ rather than just the Laplacian, and as a result produce lifts whose weak limits are invariant under a maximal $\R$-split torus of $G$ rather than just the geodesic flow.  They are most naturally thought of as distributions on $C^\infty_0(\Gamma \backslash G)$, but the standard lift may be recovered from them in the large eigenvalue limit via a correspondence between $\Gamma \backslash G$ and $S^*Y$.

For $s \in L^2(Y,B)$, the correct generalisation of definition (\ref{stdlift}) is obtained by replacing $\Op(a)$ with a pseudodifferential endomorphism of $B$ (which we shall require to be compactly supported in order to deal with the noncompactness of $Y$), so the lift $\nu_s$ should be a distribution on the corresponding space of symbols which is $C^\infty_0( S^*Y, \pi^*(E) )$.  We will define $\nu_s$ coordinatewise using the distribution $\mu_\pi (f, \Phi)$ introduced by Silberman and Venkatesh in \cite{SV}, whose definition we now recall.  If $\pi$ is the automorphic representation associated to $s$, and $f \in W_K$ and $\Phi \in W_K'$, define the functional $\mu_\pi (f, \Phi)$ on $C^\infty_{0,K}(X)$ by the rule

\begin{equation}
\label{akshaydef}
\mu_\pi( f, \Phi)(g) = \Phi \circ R_\pi^{-1} \circ P( R_\pi(f) \cdot g ),
\end{equation}

where $P: L^2(X) \rightarrow R_\pi(W)$ is the orthogonal projection, and $R_\pi(f) \cdot g$ denotes pointwise multiplication of functions on $X$.  We may now state our definition.

\begin{definition}
\label{liftdef}

Suppose $s \in L^2(Y,B)$ is an automorphic section with associated representation $\pi$.  If $f \in C^\infty_K( M \backslash K )$ and $\Phi \in W_K'$ we define the microlocal lift $\nu_s(f, \Phi)$ to be the element

\begin{equation}
\label{liftdef1}
\nu_s(f, \Phi) = \sum_{i=0}^m \mu_\pi( f \cdot \psi_{m - 2i}^*, \Phi) v_{-k}^* \otimes v_{m-2i} 
\end{equation}

of $C^\infty_{0,K} ( X, \sigma )'$, which may also be thought of as being in $C^\infty_0( S^*Y, \pi^*(E) )'$.  Here $\psi^*_i$ are as in (\ref{psi}).

\end{definition}

We then define the lift $\nu_s$ to be $\nu_s(1,\delta)$, where $\delta$ is the delta distribution at the identity in $K$.  Its properties are summarised in the following proposition.

\begin{prop}
\label{liftprops}
Let $\{ s_n \}$ be a sequence of automorphic sections of $B$ with fixed weight $k$ and spectral paramter tending to $\infty$.  Then, after replacing $\{ s_n \}$ by an appropriate subsequence and denoting $\nu_{s_n}$ by $\nu_n$, there exist sections $\ts^1_n$ and $\ts^2_n$ in $L^2(S^*Y, \pi^*(B) )$ such that

\begin{enumerate}
\item The projection of $\nu_n$ to $Y$ coincides with the element $s_n^* \otimes s_n$ of $C^\infty_0( Y, E)'$.

\item For every $f \in C^\infty_0( S^*Y, \pi^*(E))$ we have $\underset{n \rightarrow \infty}{\lim} ( \nu_n(f) - \ts^{1 \, *}_n \otimes \ts^2_n (f) ) = 0$.

\item Every weak-* limit of the measures $\ts^{1 \, *}_n \otimes \ts^2_n$ is $A$-invariant.

\item $\langle \Op(a) s_n, s_n \rangle = \nu_n(a) + o(1)$ for all $a \in C^\infty_0( S^*Y, \pi^*(E))$.

\item Let $T \subset \End_G ( C^\infty( X, \rho ) )$ be a $\C$ subalgebra of bounded automorphisms of $C^\infty( X \times \rho )$ commuting with the $G$ action and with the right action of $K$ on $X \times \rho$.  Then each $t \in T$ induces an automorphism of $C^\infty(Y,B)$, and we may suppose that $s_n$ is an eigenfunction of $T$.  Then we may choose $\ts^1_n$ and $\ts^2_n$ to be eigenfunctions with the same eigenvalues as $s_n$.

\end{enumerate}

\end{prop}

The proof of this proposition is contained in sections \ref{limits} and \ref{agreement0}, and is valid for any finite volume hyperbolic 3-manifold with the definition of automorphic section relaxed to mean one associated to a representation of $SL(2,\C)$ rather than the full Adele group.  It is similar to the analogous result of Silberman and Venkatesh in \cite{SV} and we shall follow their method of proof closely, occasionally referring the reader to their paper when our proof of a proposition is sufficiently similar to theirs.  As $\ts^{1 \, *}_n \otimes \ts^2_n \in L^1( S^*Y, \pi^*( E^* ) )$, (2) implies that any weak limits of $\{ \nu_n \}$ are measures, and (4) is the same as saying that $\nu_s$ has the characteristic property of the standard lift in the large eigenvalue limit.  The equivariance property (5) will not be relevant for us but we state it anyway; its proof is identical to that of the analogous statement in \cite{SV}, and we will not reproduce it here.  

The most natural question one may ask about the vector valued lifts we have defined is what their high energy limits should be, and simple heuristics given in section \ref{heuristics} lead us to conjecture the following answer in generalisation of QUE for functions.

\begin{conj}
\label{nonsphQUE}
If $\{s_n\}$ is a sequence of $L^2$ normalised automorphic sections of $X \times_K \rho_m$ of weight $k$, the unique quantum limit of their microlocal lifts $\nu_n$ is $v_{-k}^* \otimes v_{-k} dx / \text{Vol}(X)$.
\end{conj}

Our main theorem provides support for this conjecture, by relating it to a subconvex bound for a triple product $L$ function in the eigenvalue aspect.

\begin{theorem}
\label{main}
Let $\pi_n$ be the automorphic representations associated to $s_n$, with spectral parameters $r_n$.  Conjecture \ref{nonsphQUE} is equivalent to the asymptotics

\begin{eqnarray}
\label{ldecaypi}
\frac{ L( 1/2, \pi_n \otimes \pi_n \otimes \pi' ) }{ L( 1, \Sym^2 \pi_n )^2 } & = & o( r_n^2 ) \\
\label{ldecaychi}
\text{and} \quad \frac{ L( 1/2, \pi_n \otimes \pi_n \otimes \chi ) }{ L( 1, \Sym^2 \pi_n ) } & = & o( r_n )
\end{eqnarray}

for all automorphic representations $\pi'$ of $GL_2$ and $\chi$ of $GL_1$ which are unramified at all finite places.
\end{theorem}

As the analytic conductors of the $L$ functions occurring here are $r_n^8$ and $r_n^4$ (ignoring other parameters), the bounds of theorem \ref{main} represent modest savings over the convexity bound.  A consequence of this is that the GRH for the triple product $L$ function implies the equidistribution of $\nu_n$ at the optimal rate, which is $\nu_n(s) = O( r_n^{-1+\epsilon} )$ for $s \in C^\infty_0( X, \sigma )$ of mean 0.  It also illustrates that the phenomenon studied by Mili\'cevi\'c \cite{Mi} of base change forms becoming large at CM points of $Y$ is not strong enough to affect their global equidistribution.  It should be noted that, because our lifts are Hecke equivariant and have $A$ invariant limits, it is likely that the ergodic techniques of Lindenstrauss and Silberman and Venkatesh can be used to establish conjecture \ref{nonsphQUE} unconditionally.

We shall prove theorem \ref{main} by evaluating $\nu_n$ against $s \in C^\infty( X, \sigma )$ of the form $\phi v$, where $v \in \sigma$ and $\phi$ is a $K$-finite vector in an automorphic representation $\pi'$, using the Rankin-Selberg formula and the triple product formula of Ichino.  We formulate the required triple product integrals in section \ref{tripleprod}, before calculating the necessary asymptotics for Archimedean local integrals needed to make Ichino's formula explicit in section \ref{archimedean}.  We finish section \ref{archimedean} with a closely related calculation in the weight aspect which will be needed in a paper on QUE for automorphic forms of cohomological type \cite{Ma}, and give an interpretation of conjecture \ref{nonsphQUE} in terms of tensors and sections of local systems on $Y$ in section \ref{modQUE}.

\section{Weak Limits and $A$-invariance}
\label{limits}

In this section we verify the first three properties of our microlocal lift stated in proposition \ref{liftprops}.  We begin by simplifying the definition (\ref{liftdef1}) of $\nu_s$.  Note that if $f$ and $g$ are $K$-finite then $R_\pi(f) \cdot g$ also is, and hence if $\Phi$ may be represented as $\langle \cdot, \Phi' \rangle$ where $\Phi'$ is an infinite formal sum of $K$-types the expression

\begin{equation*}
\langle  R_\pi^{-1} \circ P( R_\pi(f) \cdot g ), \Phi' \rangle
\end{equation*}

is well defined and agrees with $\mu_\pi( f, \Phi )(g)$.  As a result we may define $\nu_s$ via the alternate expression

\begin{equation}
\label{nusimple1}
\nu_s = \sum_{i=0}^m \overline{ R_\pi(\delta) } R_\pi( \psi_{m-2i}^* ) v_{-k}^* \otimes v_{m-2i},
\end{equation}

whose integral against $g$ will reduce to a finite sum agreeing with $\nu_s(g)$ for $g \in C^\infty_{0,K} ( X, \sigma )$, and by (\ref{section2}) this may be simplified to

\begin{equation}
\label{nusimple}
\nu_s =  \overline{ R_\pi(\delta) } v_{-k}^* \otimes s.
\end{equation}

To verify (1) of proposition \ref{liftprops}, we calculate $\nu_s$ on a section $g$ of $E = X \times_K \sigma$.  As $g$ is $K$-finite we may do this using (\ref{nusimple}).

\begin{eqnarray*}
\nu_s(g) & = & \int_X \langle g, R_\pi(\delta) v_{-k} \otimes s^* \rangle dx \\
 & = & \int_X \int_K \langle \rho(k)g(xk), R_\pi(\delta) v_{-k} \otimes s^* \rangle dk dx \\
 & = & \int_X \int_K \langle g(x), \rho(k^{-1}) \left[ R_\pi(\delta)(x k^{-1}) v_{-k} \otimes s^*(x k^{-1}) \right] \rangle dk dx \\
 & = & \int_X \langle g(x), \left( \int_K  R_\pi(\delta)(x k^{-1}) \rho(k^{-1}) v_{-k} dk \right) \otimes s^*(x) \rangle dx
\end{eqnarray*}

We may now simplify the integral over $K$ as follows:

\begin{eqnarray}
\notag
& & \int_K  R_\pi(\delta)(x k^{-1}) \rho(k^{-1}) v_{-k} dk \\
\notag
& \quad & \quad \quad = \int_K  R_\pi(\rho(k^{-1}) \delta)(x) \sum_{t=0}^m v_{m-2t} \langle \rho(k^{-1}) v_{-k}, v_{m-2t} \rangle dk \\
\label{projection}
& \quad & \quad \quad = \sum_{t=0}^m v_{m-2t} R_\pi \left( \int_K \langle \rho(k^{-1}) v_{-k}, v_{m-2t} \rangle \rho(k^{-1}) \delta \, dk \right).
\end{eqnarray}

We have

\begin{eqnarray*}
\int_K \langle \rho(k^{-1}) v_{-k}, v_{m-2t} \rangle \rho(k^{-1}) \delta \, dk & = & \langle \rho(k^{-1}) v_{-k}, v_{m-2t} \rangle \\
 & = & \langle \rho(k) v_{m-2t}^*, v_{-k}^* \rangle \\
 & = & \psi_{m-2t}^*
\end{eqnarray*}

as elements of $W$, so (\ref{projection}) becomes

\begin{equation*}
\sum_{t=0}^m v_{m-2t} R_\pi (\psi_{m-2t}^*) = s
\end{equation*}

and $\nu_s(g) = \int_X \langle g, s \otimes s^* \rangle dx$ as required.

We now prove (2) and (3) of proposition \ref{liftprops}, which deal with weak limits of the $\nu_n$.  The proof of (2) is based on comparing $\nu_n(1,\delta)$ with $\nu_n( f_n, T(g_n) )$ where $\{f_n\}$ and $\{g_n\}$ are $L^2$-bounded `$\delta$-sequences' in $C^\infty_K(M \backslash K)$ and $W_K$ respectively.  We introduce the notation $\nu_s^T(f_1, f_2) = \nu_s( f_1, T(f_2) )$ ($f_1 \in C^\infty_K(M \backslash K)$, $f_2 \in W_K$) for these distributions; the following lemma shows that they are in fact finite measures, from which we will be able to deduce that limits of $\nu_n(1,\delta)$ also are.

\begin{lemma}
\label{bound}
Suppose $f_1 \in C^\infty_K(M \backslash K)$, $f_2 \in W_K$.  Then

\begin{equation*}
\nu^T_s(f_1,f_2)(g) = \int_X \sum_{i=0}^m R_\pi( f_1 \cdot \psi_{m - 2i}^* )(x) \overline{ R_\pi( f_2 )(x) } \langle g(x), v_{-k} \otimes v^*_{m-2i} \rangle dx
\end{equation*}

and $\nu^T_s(f_1,f_2)$ defines a $\sigma^*$ valued measure on $X$ of norm $\le (m+1)^{3/2} \| f_1 \|_2 \| f_2 \|_2 $.

\end{lemma}

\begin{proof}
The first statement follows from the definition of $\nu^T_s$.  For the second,

\begin{eqnarray*}
\nu^T_s(f_1,f_2)(g) & \le & \sup \| g(x) \|  \int_X \sum_{i=0}^m | R_\pi( f_1 \cdot \psi_{m - 2i}^* )(x) \overline{ R_\pi( f_2 )(x) } | dx \\
 & \le & \sup \| g(x) \|  \sum_{i=0}^m \| R_\pi( f_1 \cdot \psi_{m - 2i}^* ) \|_2 \| R_\pi( f_2 ) \|_2 \\
 & = & \sup \| g(x) \|  \sum_{i=0}^m \| f_1 \cdot \psi_{m - 2i}^* \|_2 \| f_2 \|_2 \\
 & \le & \sup \| g(x) \|  \sum_{i=0}^m \sup |\psi_{m - 2i}^*| \| f_1 \|_2 \| f_2 \|_2 \\
 & \le & (m+1)^{3/2} \sup \| g(x) \| \| f_1 \|_2 \| f_2 \|_2.
\end{eqnarray*}

\end{proof}

We now make the first refinement of the sequence $\{ s_n \}$ mentioned in proposition \ref{liftprops}, by passing to a subsequence  for which the distributions $\nu_n( f, \Phi)$ all weakly converge to limits which we shall denote $\nu_\infty(f, \Phi)$.  This will be implied by the following condition:

\begin{definition}
We say that a sequence $\{ s_n \}$ is conveniently arranged if the measures $\nu^T_n(f_1, f_2)$ are weakly convergent as $n \rightarrow \infty$ for all $f_1 \in C^\infty( M \backslash K)$ and $f_2 \in W_K$ .  In this situation we denote $\underset{n \rightarrow \infty}{\lim} \nu^T_n(f_1, f_2)$ by $\nu^T_\infty(f_1, f_2)$.
\end{definition}

The existence of a conveniently arranged subsequence is standard, and is demonstrated in \cite{SV}.  Assuming $\{ s_n \}$ to be conveniently arranged, the weak convergence of $\nu_n( f, \Phi)$ can be shown as follows.  Fix $f \in C^\infty( M \backslash K)$, $\Phi \in W_K'$, and $g \in C^\infty_{0,K} ( X, \sigma)$.  If $W_K^N$ is the subspace of $W_K$ whose isotypic components are $\rho_m$ for $m \le N$, define the $N$-truncation $\Phi_N$ of $\Phi$ to be the unique element of $W_K^N$ such that $\Phi$ and $T( \Phi_N)$ agree on $W_K^N$.  It follows from the definition of $\nu_s (f, \Phi)$ that if we choose $N = N(f,g)$ sufficiently large then $\nu_n( f, \Phi)(g) = \nu_n^T( f, \Phi_N)(g)$, and so the limit $\underset{ n \rightarrow \infty }{\lim} \nu_n( f, \Phi)(g)$ exists.  Consequently, we may define $\nu_\infty : W_K \times W_K' \rightarrow C^\infty_{0,K} (X \times \sigma)'$ as

\begin{equation*}
\nu_\infty(f, \Phi)(g) = \underset{ n \rightarrow \infty }{\lim} \nu_n( f, \Phi)(g),
\end{equation*}

and set $\nu_\infty = \nu_\infty( 1, \delta)$, so that $\nu_\infty$ is the weak limit of $\{ \nu_n \}$.  We now prove that $\nu_\infty$ is a measure by defining $\ts_n^1$ and $\ts_n^2$ and establishing (2) of proposition \ref{liftprops}, which relies on the following lemma from \cite{SV}.

\begin{lemma}
\label{parts}
Let $\{ s_n \}$ be conveniently arranged.  Then for any $f, f_1 \in C^\infty_K (M \backslash K)$ and $f_2 \in W_K$,  we have

\begin{equation*}
\nu_\infty^T (f_1, f \cdot f_2) = \nu_\infty^T (f_1 \cdot \overline{f}, f_2 ).
\end{equation*}

\end{lemma}

We apply this by defining $p \in W_K$ to be any function such that $p(e) = 1$, and letting $\{ f_j \} \in C^\infty_K (M \backslash K)$ be any sequence such that the measures $|f_j|^2$ are tending to the delta measure at the origin.  For a suitable subsequence $\{ f_{j_n} \}$ of $\{ f_j \}$ we define $\ts_n^1$ and $\ts_n^2$ by

\begin{eqnarray*}
\ts_n^1 & = & R_\pi( f_{j_n} p ) v_{-k} \\
\ts_n^2 & = & \sum_{i=0}^m R_\pi( f_{j_n} \psi^*_{m-2i} ) v_{m-2i},
\end{eqnarray*}

so that $\nu_n^T ( f_{j_n}, f_{j_n} p ) = \ts_n^{1\, *} \otimes \ts_n^2$.  Lemma \ref{bound} shows that $\{ \nu_n^T( f_{j_n}, f_{j_n} p ) \}$ is a bounded sequence of measures, and a quantitative form of lemma \ref{parts} shows that for a suitable choice of $\{ j_n \}$,

\begin{eqnarray*}
\underset{ n \rightarrow \infty }{\lim} \ts_n^{1\, *} \otimes \ts_n^2(g) & = & \underset{ n \rightarrow \infty }{\lim} \nu_n^T( f_{j_n}, f_{j_n} p )(g) \\
& = & \underset{ n \rightarrow \infty }{\lim} \nu_n^T( 1, |f_{j_n}|^2 p )(g) \\
& = & \underset{ n \rightarrow \infty }{\lim} \nu_n( 1, \delta)(g) \\
& = & \nu_\infty(g).
\end{eqnarray*}

Therefore $\nu_\infty$ is a measure as required.

We now establish the $A$-invariance of $\nu_\infty$.  We do this by showing that the co-ordinates of $\nu_n$, which are the distributions $\mu_n( \psi^*_j, \delta)$ defined in (\ref{akshaydef}), satisfy the differential equation

\begin{equation}
\label{hderiv}
H \mu_n( \psi^*_j, \delta) = \frac{1}{r_n} \mu_{j,n}',
\end{equation}

where $\{ \mu_{j,n}' \}$ is a weakly convergent sequence of distributions and $r_n$ are the spectral parameters of $s_n$, from which it follows that their weak limits satisfy $H \mu_\infty( \psi^*_j, \delta) = 0$.  $\mu_n$ is a $G$-equivariant map from $W_K \otimes W_K'$ to $C^\infty_{0,K}(X)'$, so we may proceed by establishing an equation similar to (\ref{hderiv}) for $H( \psi^*_j \otimes \delta)$.  We do this starting from the equation

\begin{eqnarray}
\label{casimir}
C \psi_j^* & = & (-r^2 - 1 + m^2/4) \psi_j^* \\
\notag
 & = & \alpha \psi_j^*
\end{eqnarray}

for the action of the Casimir operator on $I_\lambda$, and the action of $X \in \Aa \oplus \mathfrak{n}$ on $\delta$ under the dual representation, which is $\pi'(X) \delta = -(\lambda + \rho)(X) \delta$.  It follows from this that

\begin{equation}
\label{delta}
(X + (\lambda + \rho)(X) )( f \otimes \delta) = (Xf) \otimes \delta
\end{equation}

for $X \in \Aa \oplus \mathfrak{n}$, and we may use this to convert (\ref{casimir}) to an identity between derivatives of $\psi_j^* \otimes \delta$ in which $H( \psi^*_j \otimes \delta)$ is the dominant term as $r \rightarrow \infty$.  We first define $X_\mathfrak{k} = X_+ - X_- \in \mathfrak{k}$ and $Y_\mathfrak{k} = Y_+ - Y_- \in \mathfrak{k}$, and rewrite the expression for $C$ as

\begin{equation*}
4C = H^2 -4H - T^2 + 4X_+^2 - 4X_+X_\mathfrak{k} + 4Y_+^2 - 4Y_+Y_\mathfrak{k}.
\end{equation*}

We than have

\begin{eqnarray*}
(4C - 4\alpha) \psi_j^* & = & 0\\
(H^2 -4H - T^2 + 4X_+^2 + 4Y_+^2 - 4\alpha) \psi_j^* & = & ( 4X_+ X_\mathfrak{k} + 4Y_+ Y_\mathfrak{k} ) \psi_j^*\\
\; \left[ (H^2 -4H - T^2 + 4X_+^2 + 4Y_+^2 - 4\alpha) \psi_j^* \right]  \otimes \delta & = & \left[ ( 4X_+X_\mathfrak{k} + 4Y_+Y_\mathfrak{k} ) \psi_j^* \right] \otimes \delta.
\end{eqnarray*}

As the element of $U(\g)$ acting on $\psi_j^*$ on the LHS is in $U( \Aa \oplus \mathfrak{n} )$, we may shift the differentiation from $\psi_j^*$ to $\psi^*_j \otimes \delta$ using (\ref{delta}).

\begin{eqnarray*}
& & ((H + 2ir +2)^2 -4(H + 2ir + 2) - (T + im)^2 + 4X_+^2 + 4Y_+^2 - 4\alpha) (\psi_j^* \otimes \delta) \\
& & \quad \quad \quad = 4X_+ ( X_\mathfrak{k} \psi_j^* \otimes \delta ) + 4Y_+ ( Y_\mathfrak{k} \psi_j^* \otimes \delta \\
& & ( H^2 + 4irH - T^2 - 2imT + 4X_+^2 + 4Y_+^2 ) (\psi_j^* \otimes \delta) \\
& & \quad \quad \quad = 4X_+ ( X_\mathfrak{k} \psi_j^* \otimes \delta ) + 4Y_+ ( Y_\mathfrak{k} \psi_j^* \otimes \delta ) \\
& & 4irH (\psi_j^* \otimes \delta) = ( -H^2 + T^2 + 2imT - 4X_+^2 - 4Y_+^2 ) (\psi_j^* \otimes \delta) \\
& & \quad \quad \quad + 4X_+( X_\mathfrak{k}\psi_j^* \otimes \delta ) + 4Y_+( Y_\mathfrak{k}\psi_j^* \otimes \delta ).
\end{eqnarray*}

Therefore

\begin{multline*}
H \mu_n( \psi^*_j, \delta) = \frac{1}{4i r_n} \biggl[ ( -H^2 + T^2 + 2imT - 4X_+^2 - 4Y_+^2 ) \mu_n( \psi_j^*, \delta) \\
 + 4X_+ \mu_n( X_\mathfrak{k}\psi_j^*, \delta ) + 4Y_+ \mu_n( Y_\mathfrak{k}\psi_j^*, \delta ) \biggr],
\end{multline*}

and all the distributions in brackets on the RHS can be seen to be weakly convergent.  Therefore $H \mu_\infty( \psi^*_j, \delta) = 0$ and $H \nu_\infty = 0$ as required.

\section{Agreement with the Standard Lift}
\label{agreement0}

In this section we shall prove part (4) of proposition \ref{liftprops}, which states that the lift we have defined satisfies the characterising property of the standard lift, or that

\begin{equation}
\label{agreement}
\langle \Op(a) s_n, s_n \rangle = \nu_n(a) + o(1)
\end{equation}

for all smooth symbols $a \in C^\infty_0( S^*Y, \pi^*(E) )$.  We will prove this for a subspace of symbols corresponding to differential operators on $B$, homogenised by a suitable power of the Laplacian, and infer the result for all smooth symbols by density.

Let $\tau \in C^\infty_{0,K} ( X, \sigma )$, and $u \in U(\g)$ be of degree $d$.  $\tau$ gives an operator $\text{mult}_\tau$ on $C^\infty( X, \rho)$, and we will let $u$ operate on $C^\infty( X, \rho)$ by coordinate-wise differentiation.  Let $\pi^*$ be the natural map $C^\infty( Y, B) \rightarrow C^\infty( X, \rho)$ and $\pi_*$ the map $C^\infty( X, \rho) \rightarrow C^\infty( Y, B)$ given by integrating over the principal action (\ref{prin}) of $K$.  We shall homogenise our differential operators by multiplication by $(Y - C/4)^{-d/2}$, where $C$ is the Casimir operator and where $Y$ is chosen to make $Y-C/4$ positive.  In what follows we shall use $\rho$ to denote the action of $K$ on both representations $\rho$ and $\sigma$; we trust this will not lead to confusion.

We shall calculate the action of the operator $\text{MyOp}(\tau)$ defined by

\begin{equation*}
\text{MyOp}(\tau) : f \mapsto \pi_* \circ \text{mult}_\tau \circ u \circ (Y-C/4)^{-d/2} \circ \pi^* f.
\end{equation*}

If we think of $u$ as defining a polynomial $u_d$ of degree $d$ on $\g^*$ and let $\xi = 4i H^*$, its symbol will be the following section of $\pi^*(E) = X \times_M \sigma$:

\begin{equation*}
a_{\tau, u}(x) = \int_K  u(k^{-1} \xi k) \rho(k) \tau(xk) dk.
\end{equation*}

It follows that

\begin{equation*}
\langle \text{MyOp}(\tau) s_n, s_n \rangle - \langle \Op(a_{\tau,u}) s_n, s_n \rangle = o(1).
\end{equation*}

We may calculate $\langle \text{MyOp}(\tau) s_n, s_n \rangle$ explicitly; it is given by

\begin{equation}
\label{myopaction}
\langle \text{MyOp}(\tau) s_n, s_n \rangle = (Y - 1/4 - \langle \lambda_n, \lambda_n \rangle )^{-d/2} \langle \text{mult}_\tau (u s_n), s_n \rangle,
\end{equation}

and we must show that this equals $\nu_n(a_{\tau,u}) + o(1)$.  As before, the $K$-finiteness of $a_{\tau,u}$ implies that we may replace $\nu_n(a_{\tau,u})$ with $\nu^T_n(1, \delta_N)(a_{\tau,u})$ for $N$ sufficiently large, and calculate the limit of this as $N \rightarrow \infty$ as some function of $n$.

\begin{eqnarray}
\notag
\nu^T_n(1,\delta_N)(a_{\tau,u}) & = & \left\langle a_{\tau,u}, R_\pi( \delta_N ) v_{-k} \otimes s_n^* \right\rangle \\
\notag
& = & \int_X \left\langle \int_K u_d( k^{-1} \xi k ) \rho(k) \tau(xk) dk, R_\pi( \delta_N ) (x) v_{-k} \otimes s_n^* (x) \right\rangle dx\\
\notag
& = & \int_X \int_K \langle u_d( k^{-1} \xi k ) \tau(x), \rho(k^{-1}) \left[ R_\pi( \delta_N ) (xk^{-1}) v_{-k} \otimes s_n^* (xk^{-1}) \right] \rangle dk dx\\
\notag
& = & \int_X \int_K \langle u_d( k^{-1} \xi k ) \tau(x), \left[  R_\pi (\delta_N)(xk^{-1}) \rho(k^{-1}) v_{-k} \right] \otimes s_n^*(x) \rangle dk dx\\
\label{agreement1}
& = & \langle \tau, J \otimes s_n^* \rangle dx,
\end{eqnarray}

where $J(x)$ is given by

\begin{equation*}
J(x) = \int_K \overline{ u_d( k^{-1} \xi k ) } R_\pi (\delta_N)(xk^{-1}) \rho(k^{-1}) v_{-k} dk.
\end{equation*}

By expanding $\rho(k^{-1}) v_{-k}$ in the basis $\{ v_t \}$ and using the linearity and $G$-equivariance of $R_\pi$, this may be simplified as follows:

\begin{eqnarray*}
J(x) & = & \int_K \overline{ u_d( k^{-1} \xi k ) } R_\pi (\delta_N)(xk^{-1}) \rho(k^{-1}) v_{-k} dk\\
& = & \int_K \overline{ u_d( k^{-1} \xi k ) }  R_\pi (\delta_N)(xk^{-1}) \sum_{t = 0}^m \langle \rho(k^{-1}) v_{-k}, v_{m - 2t} \rangle v_{m - 2t} dk\\
& = & \sum_{t = 0}^m R_\pi \left( \int_K \overline{ u_d( k^{-1} \xi k ) } \langle \rho(k^{-1}) v_{-k}, v_{m - 2t} \rangle \rho(k^{-1})\delta_N dk \right)(x)  v_{m - 2t}.
\end{eqnarray*}

By the weak convergence of $\delta_N$, the argument of $R_\pi$ above tends to the function 

\begin{equation*}
\overline{ u_d( k^{-1} \xi k ) } \langle \rho(k^{-1}) v_{-k}, v_{m - 2t} \rangle = \overline{ u_d( k^{-1} \xi k )  }\psi^*_{m-2t}
\end{equation*}

in the $L^2$ norm of $W$.  Therefore for a suitable sequence $\epsilon_n \in L^2(Y,B)$ tending to 0 we have

\begin{equation*}
J(x) = \sum_{t = 0}^m R_\pi \left( \overline{ u_d( k^{-1} \xi k ) } \psi^*_{m-2t} \right)(x)  v_{m - 2t} + \epsilon_n,
\end{equation*}

and inserting this into (\ref{agreement1}) gives

\begin{equation*}
\nu^T_n(1,\delta_N)(a_{\tau,u}) = \left\langle \tau, \sum_{t = 0}^m R_\pi \left( \overline{ u_d( k^{-1} \xi k ) } \psi^*_{m-2t} \right)  v_{m - 2t} \otimes s_n^* \right\rangle + o(1).
\end{equation*}

As in lemma \ref{parts} we may now shift the factor $\overline{ u_d( k^{-1} \xi k ) }$ from the first factor of the tensor product to the second to obtain

\begin{equation}
\label{agreement2}
\nu^T_n(1,\delta_N)(a_{\tau,u}) = \left\langle \tau, s_n \otimes \sum_{t = 0}^m \overline{ R_\pi \left(  u_d( k^{-1} \xi k ) \psi^*_{m-2t} \right) }  v_{m - 2t}^* \right\rangle + o(1).
\end{equation}

The final step in transforming this expression to (\ref{myopaction}) is to apply the following lemma, taken from \cite{SV}, which allows us to bring the factor $u_d( k^{-1} \xi k )$ outside $R_\pi$ as the differential operator $u$.

\begin{lemma}
\label{differentiate}
Suppose that $\lambda_n = ir_nH^* + ikT^*$ with $k$ fixed and $r_n \rightarrow \infty$.  If $u \in U(\g)$ is any differential operator, let $u_d$ be the associated polynomial function of degree $d$ on $\g^*$ and $\xi \in \mathfrak{h}^*$ be as defined above.  For $f \in W_K$, we have

\begin{equation*}
\frac{1}{ \| \lambda_n \|^d } ( I_\lambda(u)f )(k) = u_d( k^{-1} \xi k ) f(k) + O( \| \lambda_n \|^{-1} )
\end{equation*}

as $n \rightarrow \infty$, where $\| \lambda \| = | \langle \lambda, \lambda \rangle |^{1/2}$.

\end{lemma}

Equation (\ref{agreement2}) now becomes

\begin{eqnarray*}
\nu^T_n(1,\delta_N)(a_{\tau,u}) & = & \frac{1}{ \| \lambda_n \|^d } \left\langle \tau, s_n \otimes \sum_{t = 0}^m \overline{ u R_\pi \left( \psi^*_{m-2t} \right) }  v_{m - 2t}^* \right\rangle + o(1) \\
& = & \frac{1}{ \| \lambda_n \|^d } \left\langle \tau, s_n \otimes ( u s_n)^* \right\rangle + o(1) \\
& = & ( Y - 1/4 - \langle \lambda_n, \lambda_n \rangle )^{-d/2} \left\langle \text{mult}_\tau (u s_n), s_n \right\rangle + o(1) \\
& = & \langle \text{MyOp}(\tau) s_n, s_n \rangle + o(1),
\end{eqnarray*}

which concludes our proof of (\ref{agreement}) for the symbol $a_{\tau,u}$.  As symbols of the form $a_{\tau,u}$ are dense in the space of all symbols in the $C^\infty_0$ topology, it follows that $\nu_\infty$ satisfies (\ref{agreement}) for all $a \in C^\infty_0(S^*Y, \pi^*(E))$ so the proof of proposition \ref{liftprops} is complete.

\subsection{Motivation of Conjecture \ref{nonsphQUE} }
\label{heuristics}

We conclude this section by giving the heuristics behind conjecture \ref{nonsphQUE}.  Let $\{ s_n \}$ be a sequence of automorphic sections of weight $k$ with microlocal lifts $\{ \nu_n \}$, and let

\begin{equation*}
a = \sum_{i,j = 0}^m f(m-2i, m-2j) v_{m-2i} \otimes v_{m-2j}^*
\end{equation*}

be an element of $C^\infty_{0,K}( S^*Y, \pi^*(E) )$.  We shall calculate $\nu_n(a)$ using the formal expansion of $\delta$.  For $l \ge |k|$, let $u_l^*$ be the vector of weight $k$ in $\rho_l^*$ and define $\beta_l = \langle \rho(k) u_l^*, u_l^* \rangle$.  Then

\begin{equation*}
\delta = \sum_{l = |k|}^\infty (l+1) T( \beta_l),
\end{equation*}

and $\nu_n(a)$ may be evaluated by the following (finite) sum:

\begin{equation*}
\nu_n(a) = \sum_{l = |k|}^\infty \sum_{i=0}^m \int_X \overline{ R_\pi( (l+1)\beta_l ) } R_\pi( \psi_{m-2i}^* ) f( -k, m-2i ) dx.
\end{equation*}

If $\psi_{m-2i}^* = \beta_l$, that is $l = m$ and $m-2i = -k$, we expect 

\begin{eqnarray*}
\underset{ n \rightarrow \infty }{ \lim }^* \; \overline{ R_\pi( (m+1)\beta_m ) } R_\pi( \psi_{m-2i}^* ) & = & (m+1) \underset{ n \rightarrow \infty}{\lim}^* \; | R_\pi( \beta_m ) |^2 \\
& = & (m+1) \| R_\pi( \beta_m ) \|^2 / \text{Vol}(X) \\
& = & 1 / \text{Vol}(X),
\end{eqnarray*}

while all other terms should be tending weakly to 0 (here $\lim^*$ denotes weak limit of functions on $X$).  This would imply that

\begin{equation*}
\underset{ n \rightarrow \infty}{\lim} \; \nu_n(a) = \frac{1}{\text{Vol}(X)} \langle f( -k, -k), 1 \rangle,
\end{equation*}

which is the assertion of conjecture \ref{nonsphQUE}.

\section{Formulation of Triple Product Integrals}
\label{tripleprod}

Having established the basic properites of our lift, we now turn to the proof of theorem \ref{main}.  By Weyl's criterion, the equidistribution of $\nu_n$ is equivalent to the convergence of each of its co-ordinates to the expected value when evaluated on a `basis' of $L^2(X)$ consisting of the constant function, cusp forms, and unitary Eisenstein series.  $\nu_n$ was defined in (\ref{nusimple1}) to be

\begin{eqnarray}
\notag
\nu_n & = & \sum_{i=1}^m \nu_{n,i} v_{-k}^* \otimes v_{m-2i} \\
\label{coords}
\nu_{n,i} & = & \overline{ R_n (\delta) } R_n( \psi_{m-2i}^* ),
\end{eqnarray}

and if $\phi \in \pi'$ is a $K$ finite vector in an automorphic representation, we may replace $\delta$ in (\ref{coords}) with a sufficiently large truncation $\delta_N$ to obtain an expression for $\nu_{n,i}(\phi)$ as a well defined integral

\begin{equation*}
\nu_{n,i}(\phi) = \int_X \overline{ R_n( \delta_N ) } R_n( \psi_{m-2i}^* ) \phi \, dx.
\end{equation*}

As in section \ref{heuristics}, the expected limit of $\nu_{n,i}$ is determined by its value on the constant function, and its convergence to this limit is equivalent to the decay of its values on all nontrivial $\pi'$.   In this section we shall use the triple product formulas of Ichino and Rankin-Selberg to express the integral defining $\nu_{n,i}(\phi)$ as the product of a central $L$ value and a local Archimedian integral, and theorem \ref{main} will follow from this once we have established asymptotics for these local integrals in section \ref{archimedean}.

\subsection{Notation}

In sections \ref{tripleprod} and \ref{archimedean} we shall change perspective slightly and think of our automorphic forms as being on $GL(2,\C)$ with trivial central character, which is equivalent to our previous definition but agrees better with the statement of Ichino's formula and the formulation of local integrals.  Let $G = GL(2,\C)$ with centre $Z$, and let $\overline{G} = G/Z$.  We define $Z$, $A$ and $N$ to be the usual subgroups of $G$, with the parameterisations

\begin{equation*}
z(t) = \left( \begin{array}{cc} t & 0 \\ 0 & t \end{array} \right), \quad a(y) = \left( \begin{array}{cc} y & 0 \\ 0 & 1 \end{array} \right), \quad n(x) = \left( \begin{array}{cc} 1 & x \\ 0 & 1 \end{array} \right).
\end{equation*}

Let $dy$ be Lebesgue measure on $\C$, and choose $dy^\times = |y|^{-2} dy$ as Haar measure on $\C^\times$.  We give $Z$, $A$ and $N$ the Haar measures $\tfrac{1}{2\pi} dt^\times$, $\tfrac{1}{2\pi} dy^\times$ and $dx$ respectively, and give the Borel the left Haar measure $db = |y|^{-2} \tfrac{1}{2\pi} dt^\times dx \tfrac{1}{2\pi} dy^\times$.  We choose the Haar measure $dg$ on $G$ to be $db dk$, where $dk$ is a Haar probability measure on $K$.  If $d\overline{b}$ is the measure $|y|^{-2} dx \tfrac{1}{2\pi} dy^\times$ on $NA = B/Z$, we choose the Haar measure $d\overline{g}$ on $\overline{G}$ to be $d\overline{b} dk$ where $dk$ is again a probability measure.  $d\overline{g}$ has the property that its pushforward to $\overline{G} / K = \HH^3$ is the standard hyperbolic measure.  Throughout, we will use the standard additive character $\psi(z) = \exp( 2\pi i \tr(z) )$ of $\C$.

Let $I_n$ be the local factors of $\pi_n$ at infinity, and $R_n$ the unitary embeddings $I_n \rightarrow L^2(X)$.  When working with various triple product formulas, it will be necessary to commute various complex conjugations past the embeddings $R_n$ or the formation of matrix coefficients.  If $v \in \pi_n$, $\overline{ R_n(v) }$ will lie in the contragredient representation $\widetilde{\pi}_n$ whose local factors are dual to those of $\pi_n$.  However, as $\pi_n$ had trivial central character all its local factors are self dual.  Therefore $\widetilde{\pi}_n = \pi_n$, and so we may let $\sigma : I_n \rightarrow I_n$ be conjugate linear isomorphisms which satisfy $\overline{ R_n(v) } = R_n( \sigma(v) )$.

\subsection{The Cuspidal Case}
\label{formulatecusp}

Let $\pi'$ be cuspidal of weight $k'$ and spectral paramteter $r'$, and let $\phi \in \pi'$ be $K$-finite.  We shall evaluate the integral

\begin{eqnarray}
\notag
\nu_{n,i}(\phi) & = & \int_X \overline{ R_n( \delta_N ) } R_n( \psi_{m-2i}^* ) \phi \, dx \\
\label{preprod}
& = & \int_X  R_n( \sigma(\delta_N) ) R_n( \psi_{m-2i}^* ) \phi \, dx
\end{eqnarray}

using a formula of Ichino \cite{I}, which we state below in our simple case in which all three automorphic forms are unramified and occur on a split form of $GL_2$.

\begin{theorem}
Notations as above (in particular for $F$ of class number one), let $\pi_i$ be three automorphic representations on $GL(2,\OO) Z \backslash GL(2,\C)$ with archimedean factors $I_i$.  Let $\phi_i \in \pi_i$ be three $L^2$ normalised $K$-finite vectors, and $v_i \in I_i$ the corresponding archimedean vectors.  Then there is a constant $C$ depending only on $F$ such that

\begin{multline}
\left| \int_X \phi_1 \phi_2 \phi_3 dx \right|^2 = C \int_{ \overline{G} } \langle I_1(g)v_1, v_1 \rangle \langle I_2(g)v_2, v_2 \rangle \langle I_3(g)v_3, v_3 \rangle d\overline{g} \\
\times \frac{ L( 1/2, \pi_1 \otimes \pi_2 \otimes \pi_3 ) }{ \prod L( 1, \sym^2 \pi_i ) }.
\end{multline}

\end{theorem}

If $u \in I'$ is the archimedean vector corresponding to $\phi$, we may apply Ichino's formula to (\ref{preprod}) to obtain

\begin{multline}
\label{tripleprod1}
| \nu_{n,i}(\phi) |^2 = C \int_{ \overline{G} } \langle I_n(g) \sigma( \delta_N ), \sigma( \delta_N ) \rangle \langle I_n(g) \psi_{m-2i}^*, \psi_{m-2i}^* \rangle \langle I'(g) u, u \rangle d\overline{g} \\
\times \frac{ L( 1/2, \pi_n \otimes \pi_n \otimes \pi' ) }{ L( 1, \sym^2 \pi_n )^2 L( 1, \sym^2 \pi' ) }.
\end{multline}

The gamma factors of this triple product $L$ function are

\begin{equation*}
L_\infty( \pi_n \otimes \pi_n \otimes \pi', s) = \prod_\pm \Gamma \left( s \pm ir_n \pm \tfrac{ir'}{2} + \tfrac{ |k|}{2} + \tfrac{ |k'|}{4} \right) \times \prod_\pm \Gamma \left( s \pm \tfrac{ir'}{2} + \tfrac{|k'|}{4} \right)^2.
\end{equation*}

We therefore see that the analytic conductor of $L( s, \pi_n \otimes \pi_n \otimes \pi' )$ behaves like $r_n^8$ in the eigenvalue aspect, so that (\ref{ldecaypi}) is a slight strengthening of the convex bound for this $L$ function.  To show that the decay of $\nu_{n,i}(\phi)$ follows from this we must show that the archimedean integral $\Ss$ appearing in (\ref{tripleprod1}) satsfies $\Ss \ll r_n^{-2}$ for all $i$ and $u$, while to deduce the decay (\ref{ldecaypi}) from $\nu_{i,n}(\phi) \rightarrow 0$ we need only show that $\Ss \gg r_n^{-2}$ for some choice of $i$ and $u$.  We shall do this in section \ref{archimedean} by transforming $\Ss$ to the corresponding integral which appears when $\phi$ is chosen to be an Eisenstein series, using a formula of Michel and Venkatesh \cite{MV}.

\subsection{The Eisenstein Case}

We now consider the case of $\phi$ an Eisenstein series.  Fix an integer $k'$, and for $t \in \C$ let $\chi_t : \C^\times \rightarrow \C^\times$ be the character

\begin{equation*}
\chi_t : z \mapsto (z / |z|)^{k'/2} |z|^{2it}.
\end{equation*}

As usual, set $s = 1/2 + it$.  Let $\pi'(s)$ be the representation of $G$ unitarily induced from the character

\begin{equation*}
\xi(s) : \left( \begin{array}{cc} z_1 & x \\ 0 & z_2 \end{array} \right) \mapsto \chi_t (z_1/z_2)
\end{equation*}

of the Borel.  We take a family of vectors $f(s)$ in the induced models for $\pi'(s)$ whose restrictions to $K$ are fixed, and let $E(s,g)$ be the associated Eisenstein series so that for $\text{Re}(s) > 1$,

\begin{equation*}
E(s,g) = \sum_{ \Gamma_\infty \backslash \Gamma} f(s)(\gamma g).
\end{equation*}

$\nu_{n,i}(E(s,\cdot))$ is equal to the integral

\begin{equation}
\label{eistriple}
\nu_{n,i}(E(s,\cdot)) = \int_X  \overline{ R_n( \delta_N )(g) } R_n( \psi_{m-2i}^* )(g) E(s,g) dg,
\end{equation}

and we shall calculate this by unfolding.  We let $\W_n$ be the Whittaker model of $\pi_n$ with respect to $\psi$, and we equip this model with the inner product

\begin{equation}
\label{whittakernorm}
\langle W_1, W_2 \rangle = \frac{1}{2\pi} \int_{\C^\times} \int_K W_1( a(y) k ) \overline{ W_2( a(y) k ) } dy^\times dk.
\end{equation}

Fix a unitary isomorphism between $I_n$ and $\W_n$, and let $W_{n,1}$ and $W_{n,2}$ correspond to $\psi_{m-2i}^*$ and $\delta_N$ so that $R_n( \psi_{m-2i}^* )$ and $R_n( \delta_N )$ have Fourier expansions

\begin{eqnarray*}
R_n( \psi_{m-2i}^* ) & = & \sum_{\xi \in \OO} a_{n,\xi} W_{n,1} ( a(\xi \kappa) g ) \\
R_n( \delta_N ) & = & \sum_{\xi \in \OO} a_{n,\xi} W_{n,2} ( a(\xi \kappa) g ).
\end{eqnarray*}

Here, $\kappa$ is a generator of the inverse different $\OO^*$ of $\OO$ and the Fourier coefficients $a_{n,\xi}$ satisfy

\begin{equation*}
a_{n,\xi} = a_{n,1} N\xi^{-1/2} \lambda_n(\xi)
\end{equation*}

where $\lambda_n$ are the automorphically normalised Hecke eigenvalues of $\pi_n$.  The $L^2$ normalisations of $R_n$ and $W_{n,i}$ imply that

\begin{equation*}
|a_{n,1}|^2 = \frac{4\pi}{ |D| L( 1, \sym^2 \pi_n )}.
\end{equation*}

Unfolding (\ref{eistriple}) for $\text{Re}(s) > 1$, we have

\begin{eqnarray*}
\nu_{n,i}(E(s,\cdot)) & =&  \int_{\Gamma_\infty \backslash \overline{G} }  R_n( \psi_{m-2i}^* )(g) \overline{ R_n( \delta_N )(g) } f(s)(g) dg \\
& = & \frac{1}{2\pi} \int_{ \C / \{ \pm 1 \} \OO } \int_{\C^\times} \int_K R_n( \psi_{m-2i}^* )( n(x) a(y) k) \overline{ R_n( \delta_N )( n(x) a(y) k) } \\
& \quad & \quad \quad f(s)( a(y) k) |y|^{-2} dx dy^\times dk.
\end{eqnarray*}

(Note that $\C / \{ \pm 1 \} \OO$ denotes the quotient of $\C \backslash \OO$ by multiplication by $-1$.)

\begin{eqnarray*}
\nu_{n,i}(E(s,\cdot)) & = & \frac{1}{4} \sqrt{|D|} \sum_{\xi \in \OO} | a_{n,\xi} |^2 \frac{1}{2\pi} \int_{\C^\times} \int_K W_{n,1}( a(\xi \kappa y) k ) \overline{ W_{n,2}( a(\xi \kappa y) k ) } \\
& \quad & \quad \quad f(s)( a(y) k) |y|^{-2} dy^\times dk \\
& = & \frac{ |a_{n,1}|^2 }{4} \sqrt{|D|} N\kappa^{1/2} \chi_t(\kappa)^{-1} \sum_{\xi \in \OO} N\xi^{-1/2} \chi_t(\xi)^{-1} | \lambda_n(\xi) |^2 \\
& \quad & \quad \quad \frac{1}{2\pi} \int_{\C^\times} \int_K W_{n,1}( a( y) k ) \overline{ W_{n,2}( a( y) k ) } f(s)( a( y) k) |y|^{-2} dy^\times dk \\
& = & \frac{ |a_{n,1}|^2 }{4} \chi_t(\kappa)^{-1} \frac{ L( 1/2, \pi_n \otimes \pi_n \otimes \chi^{-1}_t ) }{ L( \chi_t^{-2}, 1) } \\
& \quad & \quad \quad \frac{1}{2\pi} \int_{\C^\times} \int_K W_{n,1}( a( y) k ) \overline{ W_{n,2}( a( y) k ) } f(s)( a( y) k) |y|^{-2} dy^\times dk.
\end{eqnarray*}

Applying the normalisations of $|a_{n,1}|^2 $, this becomes

\begin{multline}
\label{tripleprod3}
\nu_{n,i}(E(s,\cdot)) = \frac{ \pi \chi_t(\kappa)^{-1} }{ |D| L( \chi_t^{-2}, 1) } \frac{ L( 1/2, \pi_n \otimes \pi_n \otimes \chi_t^{-1} ) }{ L( 1, \sym^2 \pi_n )}  \\
\times \frac{1}{2\pi} \int_{\C^\times} \int_K W_{n,1}( a(y) k ) \overline{ W_{n,2}( a(y) k ) } f( a(y) k) y^{-2} dy^\times dk.
\end{multline}

The gamma factors of the triple product $L$ function occurring here are

\begin{equation*}
L_\infty( s, \pi_n \otimes \pi_n \otimes \chi_t^{-1} ) = \Gamma \left( s + it + \tfrac{|k'|}{4} \right)^2 \prod_\pm \Gamma \left( s \pm ir_n + it + \tfrac{ |k|}{2} + \tfrac{ |k'|}{4} \right),
\end{equation*}

so that its analytic conductor behaves like $r_n^4$ in the eigenvalue aspect and the decay (\ref{ldecaychi}) represents a small saving over the convex bound for $L$.  If we denote the archimedean integral occurring in (\ref{tripleprod3}) by $\Tt$, the equivalence of $\nu_{n,i}(E(s,\cdot)) \rightarrow 0$ with (\ref{ldecaychi}) is implied by an asymptotic $\Tt \sim r_n^{-1}$ as in the cuspidal case.  In the next section we state a relation between $\Ss$ and $\Tt$ which shows that the asymptotics that we wish to prove for them are equivalent, and we shall establish both by calculating $\Tt$.

\section{Archimedean Integrals}
\label{archimedean}

There is a simple relation between $\Ss$ and $\Tt$ due to Michel and Venkatesh \cite{MV} which will be of great use to us.  To state it, let $v_i \in I_i$ be three vectors in representations of $GL(2,\C)$ with trivial central character.  Let $\W_1$ and $\W_2$ be the Whittaker models for $I_1$ and $I_2$ corresponding to $\psi$ and $\overline{ \psi}$, and let $\Ii_3$ be the induced model of $I_3$.  We equip $\W_i$ with the inner product

\begin{equation}
\label{whittakernorm2}
\langle W_1, W_2 \rangle = \int_{\R^+} W_1( a(y) ) \overline{ W_2( a(y) ) } dy^\times,
\end{equation}

(which is equal to the one defined in (\ref{whittakernorm}) ), and $\Ii_3$ with the inner product

\begin{equation*}
\langle f, f \rangle = \int_K | f(k) |^2 dk.
\end{equation*}

Fix unitary equivalences between $I_i$ and their respective models, under which $v_i$ correspond to $W_1$, $W_2$ and $f_3$.  Michel and Venkatesh then prove

\begin{prop}
\label{MVenkatesh}

\begin{equation}
\label{equivalence}
\int_{ \overline{G} } \langle g v_1, v_1 \rangle \langle g v_2, v_2 \rangle \langle g v_3, v_3 \rangle d\overline{g} = \frac{1}{4\pi} \left| \int_{\R^+} \int_K W_1( a(y) k ) W_2( a(y) k ) f_3( a(y) k) y^{-2} dy^\times dk \right|^2
\end{equation}

\end{prop}

In other words, our two integrals satisfy $\Ss = | \Tt |^2$, and we shall prove asymptotics for $\Tt$ using the formulas for $W_i$ given in Jacquet-Langlands \cite{JL} which we recall below.  Note that the upper bounds we require may also be proven quite easily by applying stationary phase to $\Ss$.

\subsection{Whittaker Functions}

Let $I$ be a representation of $GL(2,\C)$ with trivial central character, weight $k$ and spectral parameter $r$.  We will break with our established normalisation of these parameters and assume that $k \ge 0$ while $r$ may be negative.  Let $\W$ be the Whittaker model of $I$ with respect to $\psi$, and for a fixed $m \ge k$ let $W : \rho_m \rightarrow \W$ be the corresponding embedding of the $K$-representation $\rho_m$.  We shall use a formula for $W$ which is a simplification of the one given in Jacquet-Langlands \cite{JL}.  For $w \in \{ -m, -m+2, \ldots, m \}$, define

\begin{equation}
\label{whittakerV}
V_w(y) = y^{k/2+1} \sum_{(p,q) \in \Ai} C_{p,q} y^{p+q} K_{-ir +p -q -w/2}(4\pi y),
\end{equation}

where $C_{p,q}$ are nonzero constants depending only on $k$ and $m$ and $\Ai$ is the set of $(p,q)$ satisfying

\begin{eqnarray}
\label{whittakerindex}
p \ge 0, \quad q \ge 0, \quad p \ge w/2 - k/2, \quad q \ge -w/2 - k/2, \quad (m-k)/2 \ge p+q.
\end{eqnarray}

If we let

\begin{equation}
{\bf W}(y) = \sum_{j=0}^m V_{m-2j}(y) v_{m-2j},
\end{equation}

the embedding $W$ is given by

\begin{equation}
\label{whittakeremb}
W(a(y)k)(v) = C \langle \rho(k) v, {\bf W}(y) \rangle
\end{equation}

for a constant $C$ required to make $W$ unitary.  The integral $\Tt$ is therefore expressed in terms of integrals of two Bessel functions and a power of $y$, and may be calculated using the following formula of Gradshteyn and Ryzhik \cite{GR}:

\begin{equation}
\label{bessel}
\int_0^\infty y^\lambda K_\mu(y) K_\nu(y) dy = \frac{ 2^{\lambda-2} }{ \Gamma( \lambda+1 ) } \prod_\pm \Gamma \left( \frac{ 1 + \lambda \pm \mu \pm \nu }{2} \right).
\end{equation}

We shall make use of two cases in which exact form of (\ref{whittakerV}) is very simple.  If $w = m$ we have 

\begin{equation}
\label{whittakersimple1}
V_m(y) = y^{m/2+1} K_{- k/2 - ir}(4\pi y),
\end{equation}

and if $k = m$

\begin{equation}
\label{whittakersimple2}
{\bf W}(y) = C y^{k/2+1} \sum_{j=0}^k \tbinom{k}{j}^{1/2} K_{-ir - (k-2j)/2}(4\pi y) v_{k-2j}.
\end{equation}

We may use (\ref{whittakersimple1}) to calculate the constant $C$.

\begin{eqnarray*}
\langle W(v_m), W(v_m) \rangle & = & \int_0^\infty y^{m+2} | K_{- ir - k/2}(4\pi y) |^2 dy^\times \\
& = & (4\pi)^{-m -2} \int_0^\infty y^{m+1} K_{-k/2 + ir}(y) K_{-k/2 - ir}(y) dy \\
& = & (4\pi)^{-m -2} \frac{ 2^{m-1} }{ \Gamma( m + 2) } \Gamma(1 + \tfrac{m}{2} \pm \tfrac{k}{2} ) \Gamma \left( 1 + \tfrac{m}{2} \pm ir \right) \\
& = &  \frac{ \Gamma(1 + \tfrac{m}{2} \pm \tfrac{k}{2} ) \Gamma \left( 1 + \tfrac{m}{2} \pm ir \right) }{ 8\Gamma( m + 2) (2\pi)^{m +2} }.
\end{eqnarray*}

Therefore, ignoring absolute constants,

\begin{equation}
\label{whittakerunitary}
C = (2\pi)^{m/2} \frac{ \Gamma(m+2)^{1/2} }{ \Gamma(1 + \tfrac{m}{2} \pm \tfrac{k}{2} )^{1/2} } | \Gamma \left( 1 + \tfrac{m}{2} + ir \right) |^{-1}.
\end{equation}

Substituting this into (\ref{whittakersimple2}) when $k=m$, we have

\begin{eqnarray*}
{\bf W}(y) = C y^{k/2+1} \sum_{j=0}^k \tbinom{k}{j}^{1/2} K_{-ir - (k-2j)/2}(4\pi y) v_{k-2j}, \\
\text{where} \quad C = (2\pi)^{k/2} (k+1)^{1/2} | \Gamma( 1 + \tfrac{k}{2} + ir)|^{-1}.
\end{eqnarray*}

\subsection{Computation in the Eigenvalue Aspect}
\label{archeval}

In either the cuspidal or Eisenstein case, the matrix coefficient form of the integral to be evaluated is

\begin{equation*}
\Ss = \int_{ \overline{G} } \langle I_n(g) \sigma( \delta_N ), \sigma( \delta_N ) \rangle \langle I_n(g) \psi_{m-2i}^*, \psi_{m-2i}^* \rangle \langle I'(g) u, u \rangle d\overline{g}.
\end{equation*}

As (\ref{equivalence}) ensures that this integral is real, we may take its complex conjugate and use the relations $\overline{ \langle u, v \rangle } = \langle \sigma(u), \sigma(v) \rangle$ and $\sigma( \psi_{m-2i}^* ) = z \psi_{2i-m}^*$ ($|z| = 1$) to rewrite it as 

\begin{equation}
\label{tripleprod2}
\Ss = \int_{ \overline{G} } \langle I_n(g) \delta_N, \delta_N \rangle \langle I_n(g) \psi_{2i-m}^*, \psi_{2i-m}^* \rangle \langle I'(g) \sigma(u), \sigma(u) \rangle d\overline{g}.
\end{equation}

We shall replace $\psi_{2i-m}^* \in \rho_m^*$ with $v_{m-2i} \in \rho_m$, so that we can forget about taking duals of $\rho_m$, and let $w_1 = m-2i$ be its $M$-weight.  We may also assume that $u$ is an eigenvector for $M$ of weight $w_2$ in an irreducible representation $\rho_{m'}$ of $K$.  We transform (\ref{tripleprod2}) to an integral of Whittaker functions using proposition \ref{MVenkatesh} by transferring $v_{m-2i}$ and $\sigma(u)$ to their Whittaker models and $\delta_N$ to its induced model.  The induced vector corresponding to $\delta_N$ is

\begin{equation*}
f( n a(y) k ) = y^{1+ir} \delta_N(k),
\end{equation*}

and we let $W_1$ be the Whittaker embedding of $\rho_m$ in $\W(I_n, \psi)$ (we shall ignore dependencies on $n$ for the rest of the section).  Because the Whittaker function of $\sigma(u)$ with respect to $\overline{\psi}$ is equal to the conjuagate of the Whittaker function of $u$ with respect to $\psi$, we may let $W_2$ be the Whittaker embedding of $\rho_{m'}$ in $\W(I', \psi)$ so that we are left with calculating $\Tt = \int W_1(v_{m-2i}) \overline{W_2(u)} f$.  We now substitute the formula (\ref{whittakeremb}) into this integral, letting $C_1$ and $C_2$ be the appropriate constants of unitary normalisation.

\begin{eqnarray*}
\Tt & = & C_1 C_2 \int_0^\infty \int_K W_1( a(y) k)(v_{m-2i}) \overline{ W_2( a(y) k)(u)} f( a(y) k) y^{-2} dy^\times dk \\
& = & C_1 C_2 \int_0^\infty \int_K \langle \rho(k) v_{m-2i}, {\bf W}_1(y) \rangle \overline{ \langle \rho(k) u, {\bf W}_2(y) \rangle} \delta_N(k) y^{-1 + ir} dy^\times dk.
\end{eqnarray*}

As we are only interested in the asymptotic behaviour of $\Tt$ as $r \rightarrow \infty$, we may ignore all factors which are independent of it, and write $\propto$ to indicate that two quantities are proportional with a contant whose absolute value is independent of $r$.  $\delta_N(k)$ has weight $k$ under the right action of $M$, so for this integral to be nonzero we must have $w_1 - w_2 +k = 0$.  It also has weight $k$ under the left $M$ action, and its integral against functions $h(k)$ on $K$ of left weight $-k$ is $h(e)$.  Because the other term in the integral has right weight $-k$, all its components with left weight other than $-k$ must vanish at the identity.  Therefore the inner integral in $K$ reduces to evaluation of the first two terms at the identity, and our formula simplifies to

\begin{eqnarray*}
\Tt & = & C_1 C_2 \int_0^\infty \langle v_{m-2i}, {\bf W}_1(y) \rangle \overline{ \langle u, {\bf W}_2(y) \rangle} y^{-1 + ir} dy^\times \\
& = & C_1 C_2 \int_0^\infty \overline{ V_1(y) } V_2(y) y^{-1 + ir} dy^\times
\end{eqnarray*}

where $V_1$ and $V_2$ are short for $V_{1,w_1}$ and $V_{2,w_2}$.  We now expand this integral using the sums for $V_1$ and $V_2$, and consider the behaviour of each term seperately.  The two sums are

\begin{eqnarray*}
V_1(y) & = & y^{k/2+1} \sum_{(p,q) \in \Ai} C_{p,q} y^{p+q} K_{-ir +p -q -w_1/2}(4\pi y), \\
V_2(y) & = & y^{k'/2+1} \sum_{(p',q') \in \Bi} C_{p',q'} y^{p'+q'} K_{-ir' +p' -q' -w_2/2}(4\pi y),
\end{eqnarray*}

where $\Ai$ and $\Bi$ are as in (\ref{whittakerindex}).  Ignoring constant factors, the term in our integral corresponding to a given $p$, $q$, $p'$ and $q'$ is

\begin{eqnarray*}
C_1 C_2 \int_0^\infty y^{a+ir} K_{b+ir}(4\pi y) K_c(4\pi y) dy^\times 
& = & C_1 C_2 \frac{ \Gamma( \tfrac{a -b \pm c}{2} ) \Gamma( \tfrac{a +b \pm c}{2} + ir ) }{ 8 \Gamma(a + ir) (2\pi)^{a+ir} } \\
& \propto & \frac{ \Gamma( (a +b \pm c)/2 + ir ) }{ \Gamma(a + ir) |\Gamma( 1 + \tfrac{m}{2} + ir)| },
\end{eqnarray*}

\begin{eqnarray*}
\text{where} \quad a & = & 1 + (k+k')/2 + p + q + p' + q', \\
b & =&  p- q- w_1, \\
\text{and} \quad c & = & -ir' + p' - q' - w_2.
\end{eqnarray*}

By Stirling's formula, the asymptotic behaviour of this expression as $|r| \rightarrow \infty$ will be $|r|^\sigma$, where $\sigma$ is the sum of the real parts of the arguments of the gamma function occurring in the numerator minus those in the denominator.  This is $b - m/2 - 1$, or $p -q -w_1 -m/2 -1$, and we wish to show that this is $\le -1$.  Adding the third and fourth constraints for $p$ and $q$ in (\ref{whittakerindex}) gives $m/2 + w_1/2 \ge p$, so

\begin{eqnarray}
\label{indexineq}
p - q -w_1/2 - m/2 - 1 & \le & m/2 + w_1/2 - w_1/2 - m/2 -1 \\
& = & -1,
\end{eqnarray}

and $\Tt \ll r^{-1}$ as required.  To establish the lower bound, we begin by determining those $p$ and $q$ for which equality can hold in (\ref{indexineq}).  We must have $q = 0$ and $p = m/2 + w_1/2$, so that the fourth and fifth inequalities of (\ref{whittakerindex}) become

\begin{eqnarray*}
0 & \ge & -w_1/2 - k/2 \\
w_1 & \ge & -k, \\
\text{and} \quad (m-k)/2 & \ge & m/2 + w_1/2 \\
-k & \ge & w_1.
\end{eqnarray*}

Therefore the only equality case is when $w_1 = -k$, $p = (m-k)/2$ and $q = 0$.  To exhibit the asymptotic lower bound, we may take $u \in I'$ to be the vector of minimal $K$-type and $w_2 = 0$.  Then $V_2 = V_{2,0}$ is $y^{k'/2+1} K_{-ir}(4\pi y)$, and up to factors independent of $r$ the expression for $\Tt$ is

\begin{equation*}
\Tt \propto |\Gamma( 1 + \tfrac{m}{2} + ir)|^{-1} \int_0^\infty \overline{V_1(y)} y^{k'/2 + 2 + ir} K_{-ir}(4\pi y) dy^\times.
\end{equation*}

We have shown that all terms other than $(p,q) = ( (m-k)/2, 0)$ in $V_1$ make a contribution of $O(r^{-2})$ to this integral, while $( (m-k)/2, 0)$ makes a contribution asymptotic to $r^{-1}$ because the expression for $V_2$ only contains one term.  Therefore $\Tt \gg r^{-1}$ for this choice of $\psi^*_{m-2i}$ and $u$, which concludes the proof of theorem \ref{main}.

\subsection{Computation in the Weight Aspect}
\label{tripleweight}

We finish this chapter by computing two triple product integrals which we will need for a paper on QUE in the weight aspect \cite{Ma}.  Let $I$ have weight $k \ge 0$ and spectral parameter 0, and $v_{\pm k} \in \rho_k$ be the two vectors of minimal $K$-type and extremal weight in $I$.  Let $I'$ be spherical with spectral parameter $r'$, and let $u$ be the unit $K$-fixed vector.  The first integral we wish to calculate is

\begin{equation*}
\Ss = \int_{ \overline{G} } \langle g v_k, v_k \rangle \langle g v_{-k}, v_{-k} \rangle \langle g u, u \rangle d\overline{g}.
\end{equation*}

We transfer $v_k$ to the function $f$ in its induced model, and let $W_1$ and $W_2$ be the Whittaker embeddings of $u$ and $\rho_k$ in $\W( I, \psi)$ and $\W( I', \psi)$ respectively, so that $\Ss$ is determined by the Whittaker integral $\Tt_1$ given by

\begin{equation*}
\Tt_1 = \int_0^\infty \int_K W_1( a(y)k )(u) \overline{ W_2( a(y)k )( v_k) } f( a(y)k ) y^{-2} dy^\times dk.
\end{equation*}

(Note that we are conjugating $W_2$ and replacing $v_{-k}$ with $\sigma( v_{-k} ) \simeq v_k$ as before).  $f$ is given by

\begin{equation*}
f( a(y)k ) = (k+1)^{1/2} y \langle \rho(k) v_k, v_k \rangle,
\end{equation*}

and by the formulas (\ref{whittakerV}) and (\ref{whittakerunitary}) $W_1$ and $W_2$ are

\begin{eqnarray*}
W_1( a(y) k) & = & | \Gamma( 1 + ir') |^{-1} y K_{ir'}(4\pi y), \\
W_2( a(y)k )(v) & = & \frac{ (2\pi)^{k/2} (k+1)^{1/2} }{ \Gamma( 1 + \tfrac{k}{2} ) } \langle \rho(k) v, {\bf W}_2(y) \rangle, \\
\text{where} \quad {\bf W}_2(y) & = & y^{k/2+1} \sum_{i=0}^k \tbinom{k}{i}^{1/2} K_{k/2-i}(4\pi y) v_{k-2i}.
\end{eqnarray*}

Substituting these into $\Tt_1$ gives

\begin{equation*}
\Tt_1 = \frac{ (2\pi)^{k/2} (k+1) }{ \Gamma( 1 + \tfrac{k}{2} ) | \Gamma(1+ir')| } \int_0^\infty \int_K y K_{ir'}(4\pi y) \overline{ \langle \rho(k) v_k, {\bf W}_2(y) \rangle } y \langle \rho(k) v_k, v_k \rangle y^{-2} dy^\times dk.
\end{equation*}

We may perform the integral over $K$ using the inner product formula for matrix coefficients, which gives

\begin{eqnarray*}
\Tt_1 & = & \frac{ (2\pi)^{k/2} }{ \Gamma( 1 + \tfrac{k}{2} ) | \Gamma(1+ir')| } \int_0^\infty  y^{k/2+1} K_{ir'}(4\pi y) K_{k/2}(4\pi y) dy^\times \\
& = & \frac{ (2\pi)^{k/2} }{ \Gamma( 1 + \tfrac{k}{2} ) | \Gamma(1+ir')| } (4\pi)^{-k/2-1} \frac{ 2^{k/2-2} }{ \Gamma(1 + \tfrac{k}{2} ) } \prod_\pm \Gamma \left( \tfrac{ 1 + k/2 \pm k/2 \pm ir'}{2} \right) \\
& = & \frac{ \Gamma \left( \tfrac{ 1 + k \pm ir'}{2} \right) \Gamma \left( \tfrac{ 1 \pm ir'}{2} \right) }{ \Gamma( 1 + \tfrac{k}{2} )^2 | \Gamma(1+ir')| }
\end{eqnarray*}

(ignoring the constant factor in the last step).  The second integral we shall require is

\begin{equation*}
\Tt_2 = \frac{ (2\pi)^k }{ \Gamma( k/2+1) } \int_0^\infty y^{1+ir'} |W_2(y)|^2 y^{-2} dy^\times,
\end{equation*}

and we may obtain this easily from the calculation above.  We first apply the identity

\begin{equation*}
|W_2(y)|^2 = (k+1) \int_K | \langle \rho(k) v_k, W_2(y) \rangle |^2 dk,
\end{equation*}

so that $\Tt_2$ becomes

\begin{eqnarray*}
\Tt_2 & = & \frac{ (2\pi)^k (k+1) }{ \Gamma( k/2+1) } \int_0^\infty \int_K y^{1+ir'} | \langle \rho(k) v_k, W_2(y) \rangle |^2 y^{-2} dy^\times dk \\
& = & \int_0^\infty \int_K y^{1+ir'} W_2( a(y)k )( v_k) \overline{ W_2( a(y)k )( v_k) } y^{-2} dy^\times dk.
\end{eqnarray*}

The three vectors ocurring here have norm one in their respective models, and when we convert the integral to its matrix coefficient form we obtain $\Ss$.  Therefore the absolute value of $\Tt_2$ is the same as that of $\Tt_1$.

\section{Geometric Interpretations of QUE}
\label{modQUE}

Many differential geometric objects on $Y$, such as tensors and sections of certain local systems, correspond to sections of the $K$-principal bundles with which we have been working, and in this section we shall describe these correspondences and the associated reinterpretations of conjecture \ref{nonsphQUE}.  In the case of vector fields and differential forms, if $p \in \HH^3$ corresponds to the identity under the isomorphism $G / K \simeq \HH^3$, there is a natural identification of $\p$ with $T \HH^3_p$ which may be extended by left equivariance to a map $\tilde{\phi} : G \times \p \rightarrow T\HH^3$.  It may be checked that $\tilde{\phi}$ factors through the principal action (\ref{prin}) of $K$, and hence descends to a $G$-equivariant map $\phi : G \times_K \p \rightarrow T\HH^3$ which is in fact an isomorphism.  Taking quotients on the left by $\Gamma$ gives an isomorphism $X \times_K \p \simeq TY$.  We therefore have $TY \simeq T^*Y \simeq X \times_K \rho_2$, and so all scalar valued tensors on $Y$ may be viewed as sections of $K$-principal bundles.

In the case of 1-forms, exact Laplace 1-forms correspond to automorphic representations $\pi$ of weight 0, and are of the form $df$ where $f \in \pi$ is the $K$-fixed vector.  Our conjecture then implies that quantum limits of exact forms are dual to the endomorphism of $\pi^*( T^*Y )$ whose value at $(x,v) \in S^*Y$ is the projection onto $v$.  Coclosed forms correspond to representations of weight $\pm 2$, and their quantum limits should be dual to the operator which at $(x,v)$ projects onto the vector $u \in v^\perp$ which is an eigenvector of weight $\mp 2$ for $M$.

The local systems which are isomorphic to principal bundles are those obtained by restricting a representation $\tau$ of $G$ to $\Gamma$, and will be denoted $V_\tau$.  The isomorphism may be seen by pulling  $V_\tau$ back to $\widetilde{V}_\tau$ on $X$, which is trivial as may be seen from the map

\begin{eqnarray}
\notag
\phi: G \times \tau & \rightarrow & G \times \tau \\
\label{twist}
(g, v) & \mapsto & (g, \tau(g^{-1}) v ).
\end{eqnarray}

This conjugates the $\Gamma$ action $\gamma : (g, v) \mapsto ( \gamma g, \tau(\gamma) g)$ to one which is trivial on $\tau$, and so gives an isomorphism $\widetilde{V}_\tau \simeq X \times \tau$.  $V_\tau$ is the quotient of $\widetilde{V}_\tau$ by the right action of $K$ on $G$, and this is conjugated under (\ref{twist}) to the principal action of $K$ on $X \times \tau$ so that $V_\tau \simeq X \times_K \tau$.  We may use this isomorphism to equip $V_\tau$ with a canonical positive definite inner product $\langle \;, \; \rangle$, by choosing an inner product on $\tau$ with respect to which the actions of all vectors in $\mathfrak{k}$ are skew-Hermitian and all vectors in $\p$ are Hermitian (such a product is unique up to scaling).  It is proven in \cite{BW} that with this choice of inner product, automorphic sections of $V_\tau$ are also eigensections with respect to the Laplacian induced by $\langle \;, \; \rangle$, and so we may use the isomorphism $V_\tau \simeq X \times_K \tau$ to apply our microlocal lift construction to Laplce eigensections of $V_\tau$.  If we take a series of Laplace eigensections corresponding to a fixed weight and irreducible $K$-summand of $\tau$, we then see that the expected quantum limit of these sections will be dual to the projection onto a certain $M$-eigenvector $v \in \tau$ in the bundle $X \times_M \tau$.

\end{document}